\documentclass[11pt,epsf]{article}
\usepackage[dvips]{color}
\usepackage{geometry}
\newcommand{\bed}{\begin{displaymath}}
\newcommand{\eed}{\end{displaymath}}
\newcommand{\bea}{\bed\begin{array}{rl}}
\newcommand{\eea}{\end{array}\eed}
\newcommand{\barray}{\begin{array}{ll}}
\newcommand{\earray}{\end{array}}

\usepackage{color,latexsym,amsfonts,amssymb,amsmath,amsthm}
\usepackage{colordvi,amsfonts,amsmath,epsfig,subfigure,cite}
\usepackage{graphicx}
\usepackage{caption}
\newtheorem{theorem}{Theorem}[section]
\newtheorem{lemma}[theorem]{Lemma}

\newtheorem{corollary}[theorem]{Corollary}

\newtheorem{definition}{Definition}[section]

\usepackage{fancyhdr}
\begin{document}

\title{ Existence and stability of almost periodic solutions to impulsive stochastic differential evolution equations with infinite delay  }
\author{Shufen Zhao$^{a,b}$\thanks{Corresponding author. zsfzx1982@sina.com, 12b312003@hit.edu.cn. This work is supported by the NSF of P.R. China (No.11071050)}
,~~Minghui Song$^{a}$
\\$a$ Department of Mathematics, Harbin Institute of Technology, Harbin 150001, PR China\\
$b$ Department of  Mathematics, Zhaotong University, Zhaotong 657000, PR China}
\maketitle
\begin{abstract}
In this paper, we investigate a class of nonlinear impulsive stochastic differential evolution equations with infinite delay in Banach space.  Based on the Krasnoselskii's fixed point theorem, sufficient conditions of the existence of the square mean piecewise almost periodic solutions to this type of equations are derived. Moreover, the exponential stability of the square mean piecewise almost periodic solution is investigated.

{\bf Mathematics Subject Classfication} 35B15; 34F05; 60H15.

\end{abstract}
\pagestyle{fancy}
\fancyhf{}
\fancyhead[CO]{ Almost periodic solution to impulsive stochastic differential equations with delay}
\fancyhead[LE,RO]{\thepage}
\section{Introduction}
\setcounter{equation}{0}
\noindent
Stochastic modeling has come to play an important role in describing the phenomena which are influenced by some random factors. In the past decades, the qualitative properties such as existence, uniqueness and stability for stochastic differential systems have been investigated in \cite{0,1,2,3,4,5,6} and the references therein. The concept of quadratic mean almost periodicity was introduced in \cite{7}. In \cite{7,8}, Bezandry and Diagana investigated the existence and uniqueness of a quadratic mean almost periodic solution to the stochastic evolution equations and a non-autonomous semi-linear stochastic differential equation, respectively.
It is well known that impulsive phenomena are frequently encountered in a variety of dynamic systems such as nuclear reactors, chemical engineering systems, biological systems and population dynamic models.
In \cite{9},
 Liu and Zhang introduced the concept of square mean piecewise almost periodic for impulsive stochastic process, they proved the existence and uniqueness of the square mean piecewise almost periodic solution by using the theory of the semigroups of the operators and the Schauder fixed point theorem. Meanwhile, they discussed the exponential stability of mild solution.

Li, Liu and Luo investigated the existence and uniqueness of quadratic mean almost periodic mild solutions for a class of stochastic functional evolution equations with infinite delays with the aid of semigroups of operators and the fixed point method of contracting mapping in \cite{10}.

Motivated by the above works, we consider the existence and stability of the square mean piecewise almost periodic solution of the following stochastic impulsive differential equation.
\begin{eqnarray}\label{eq1}
\left\{ \begin{array}{ll}
\mathrm{d}[x(t)-G(t,x_{t},x_{t})]=[Ax(t)+f(t,x_{t},x_{t})]\mathrm{d}t+g(t,x_{t},x_{t})\mathrm{d}w(t),\, t\in \mathbf{R},\, t\neq t_{i},\\
\triangle x(t_{i})=I_{i}(x(t_{i})),\,i=1,2,\ldots,\\
x_{\sigma}=\varphi\in\mathcal{B}.
\end{array} \right.
\end{eqnarray}
where $A:D(A)\subset H\rightarrow H$ is the infinitesimal generator of a sectorial operator. $G,\,f:\mathbf{R}\times H\times\mathcal{B}\rightarrow H$ and  $g:\mathbf{R}\times H\times\mathcal{B}\rightarrow \mathcal{L}(G,H)$  are appropriate mappings. Here $\mathcal{B}$ is an abstract phase space to be defined later. The history $x_{t}:(-\infty,0]\rightarrow H,\, x_{t}(s)=x(t+s),s\leq0$ belongs to the abstract phase space $\mathcal{B}$.
Moreover, we denotes $\triangle x(t_{i})=x(t_{i}^{+})-x(t_{i}^{-}).$
Let $x(t_{i}^{+})$ and $x(t_{i}^{-})$ represent the right and the left limits of $x(t)$ at $t=t_{i},$ respectively.

By using the Krasnoselskii-Schaefer fixed point theorem (see \cite{11}) and the theory of the semigroup of the operators (see \cite{12}), we get the existence and uniqueness of the square mean piecewise almost periodic solution. Moreover, by using the method appearing in \cite{9}, we obtain the exponential stability of the square mean piecewise almost periodic solution of (\ref{eq1}). The problem (1.1) reduced to the abstract form as
in \cite{10} when $\triangle x(t_{i})=0$ ( the impulsive  phenomenon does not exist). However, the condition that ensure the existence of the square mean piecewise almost periodic solution is different. We are not aware of any results in the literature that solves problem (\ref{eq1}), this is the first time to consider the existence and the uniqueness of the square mean piecewise almost periodic solution for (\ref{eq1}).

The remainder of this paper is organized as follows. In Section 2, we give some preliminaries which are used in this paper. In Section 3,  we give some criteria to ensure the existence and the uniqueness of the square mean piecewise almost periodic mild solution. In Section 4, we discuss the exponential stability of the square mean piecewise almost periodic mild solution.
\section{Preliminaries}
Let $(\Omega,\mathcal{F},P)$ be a complete probability space equipped with some filtration $\{\mathcal{F}_{t}\}_{t\geq0}$ satisfying the usual conditions, i.e., the filtration is right continuous and increasing while $\mathcal{F}_{0}$ contains all $P$-null sets. $H$, $G$ be two real separable Hilbert spaces. $<\cdot,\cdot>_{H}$,  $<\cdot,\cdot>_{G}$  denote the inner products on $H$ and $G$, respectively. And $|\cdot|_{H}$, $|\cdot|_{G}$  are vector norms on $H$, $G$. Let $\mathcal{L}(G,H)$ be the collection of all inner bounded operators from $G$ into $H$, with the usual operator norm $\|\cdot\|$.
The symbol $\{w(t),t\geq 0\}$ is a $G$ valued $\{\mathcal{F}_{t}\}_{t\geq0}$ Wiener process defined on the probability space $(\Omega,\mathcal{F},P)$ with covariance operator $Q$, i.e.
\begin{equation*}
    \mathbf{E}<w(t),x>_{G}<w(s),y>_{G}=(t\wedge s)<Qx,y>_{G},\:\forall x,y\in G
\end{equation*}
where $Q$ is a positive, self-adjoint and trace class operator on $K$. In particular, we regard $\{w(t),t\geq 0\}$ as a $G$ valued $Q$ wiener process related to $\{\mathcal{F}_{t}\}_{t\geq0}$ (see \cite{11,2}), and $w(t)$ is defined as
\begin{equation*}
    w(t)=\sum_{n=1}^{\infty}\sqrt{\lambda_{n}}\beta_{n}(t)e_{n},\,t\geq0,
\end{equation*}
where $\beta_{n}(t)\,(n=1,2,3,\ldots)$ is a sequence of real valued standard Brownian  motions mutually independent on the probability space $(\Omega,\mathcal{F},P)$, let $\lambda_{n}$ $(n\in \mathbf{N})$ are the eigenvalues of $Q$ and $e_{n}$ $(n\in \mathbf{N})$ are the eigenvectors of  $\lambda_{n}$ corresponding to $\lambda_{n}$. That is
\begin{equation*}
    Qe_{n}=\lambda_{n}e_{n},\,n=1,2,3,\ldots.
\end{equation*}
In order to define stochastic integrals with respect to the $Q$ wiener process $w(t)$, we introduce the subspace $G_{0}=Q^{1/2}(G)$ of $G$ which, endowed with the inner product,
\begin{equation*}
    <u,v>_{G_{0}}=<Q^{1/2}u,Q^{1/2}v>_{G}.
\end{equation*}
 is a Hilbert space. Let $\mathcal{L}_{2}^{0}=\mathcal{L}_{2}(G_{0},H)$ denotes the collection of all Hilbert Schmidt operators from $G_{0}$ into $H$. It turns out to be a separable Hilbert space equipped with the norm
\begin{equation*}
\|\psi\|^{2}_{\mathcal{L}_{2}^{0}}=tr\big((\psi Q^{1/2})(\psi Q^{1/2})^{*}\big), \quad\forall\,\psi\in \mathcal{L}_{2}^{0}.
\end{equation*}
Clearly, for any bounded operator $\psi\in \mathcal{L}(G,H)$, this norm reduces to $\|\psi\|^{2}_{\mathcal{L}_{2}^{0}}=tr(\psi Q\psi^{*})$.

Let $\Phi:(0,\infty)\rightarrow\mathcal{L}_{2}^{0}$ be a predictable and $\mathcal{F}_{t}$ adapted process such that
\begin{equation*}
    \int_{0}^{t}\mathbf{E}\|\Phi(s)\|^{2}_{\mathcal{L}_{2}^{0}}\mathrm{d}s<\infty,\, \forall t>0.
\end{equation*}
Then we can define the $H$ valued stochastic integral
\begin{equation*}
    \int_{0}^{t}\Phi(s)\mathrm{d}w(s),
\end{equation*}
which is a continuous square-integrable martingale (\cite{111}).
The collection of all strongly measurable, square-integrable and $H$-valued random variables satisfying $\int_{\Omega}\|x\|\mathrm{d}P<\infty$, which is denoted by $L_{2}(P,H),$ is a Banach space equipped with norm $\|x(\cdot)\|_{L_{2}(P, H)}=(\mathbf{E}\|x(\cdot)\|^{2}_{H})^{1/2}$ . In the following, we assume $g:\mathbf{R}\times L_{2}(P, H)\times\mathcal{B}\rightarrow L_{2}(P,\mathcal{L}_{2}^{0})$ in (\ref{eq1}).

\begin{definition}\cite{7}
A stochastic process $x : \mathbf{R}\rightarrow L_{2}(P,H)$ is said to be stochastically bounded if there exists $M>0$ such that  $\|x(t)\|_{L_{2}(P, H)}\leq M$ for all $t\in \mathbf{R}.$
\end{definition}
\begin{definition}\cite{7}
A stochastic process $x : \mathbf{R}\rightarrow L_{2}(P,H)$ is said to be stochastically continuous in $s\in\mathbf{R}$ if $\lim_{t\rightarrow s}\|x(t)-x(s)\|_{L_{2}(P, H)}^{2}=0.$
\end{definition}

Let $T$ be the set consisting of all real sequence $\{t_{i}\}_{i\in\mathbf{Z}}$ such that $\gamma=\inf_{ \in\mathbf{Z}}(t_{i+1}-t_{i})>0,$ $t_{0}=0$ and $\lim_{i\rightarrow\infty}t_{i}=\infty.$  For  $\{t_{i}\}_{i\in\mathbf{Z}}\in T,$ let $PC(\mathbf{R},L_{2}(P,H))$  be the collection of all stochastically bounded functions $\phi:\mathbf{R}\rightarrow L_{2}(P,H),$  $\phi(\cdot)$ is stochastically continuous at $t$ for any $t\in(t_{i},t_{i+1})$ and $\phi(t_{i})=\phi(t_{i}^{-})$ for all $i\in\mathbf{Z}$; let $PC(\mathbf{R}\times L_{2}(P,H)\times\mathcal{B},L_{2}(P,H))$ be the collection of all stochastic processes $\phi:\mathbf{R}\times L_{2}(P,H)\times\mathcal{B} \rightarrow L_{2}(P,H)$ such that for any $(x, \tilde{x})\in L_{2}(P,H)\times\mathcal{B},$ $\phi(\cdot,x,\tilde{x})$ is stochastic continuous at $t$ for any $t\in(t_{i}, t_{i+1})$ and $\phi(t_{i},x,\tilde{x})=\phi(t_{i}^{-},x,\tilde{x})$ for all $i\in\mathbf{Z}.$ $\phi(t,\cdot,\cdot)$ is stochastically continuous at $(x,\tilde{x})\in L_{2}(P,H)\times\mathcal{B}.$

In this paper, we assume that the phase space $\mathcal{B}$ is a linear space formed by functions mapping $(-\infty,0]$ into $L_{2}(P,H)$ which
are $\mathcal{F}_{0}$-measurable functions, with a norm $\|\cdot\|_{\mathcal{B}}$ satisfying
 \begin{enumerate}
       \item [(1)] If $x\in PC(\mathbf{R},L_{2}(P,H))$, then, for every $t\in \mathbf{R} $, $x_{t}\in \mathcal{B},$
        \item [(2)]  $\|x_{t}\|_{\mathcal{B}}=\sup_{s\leq t}\|x(s)\|_{L_{2}(P,H)},$
        \item [(3)] The space $\mathcal{B}$ is complete. If $\{\phi^{n}\}_{n\in\mathbf{N}}\subset\mathcal{B}$ is a uniformly bounded sequence in $PC((-\infty,0],L_{2}(P,H))$ formed by functions with compact support and $\phi^{n}\rightarrow\phi$ in the compact open topology, then $\phi\in \mathcal{B}$ and $\|\phi^{n}-\phi\|_\mathcal{B}\rightarrow 0,$ as $n\rightarrow\infty.$
        \end{enumerate}
\begin{definition}
A number $\tau $ is called a $\epsilon$-translation number of the function $\phi\in PC(\mathbf{R},L_{2}(P,H))$ if $ \|\phi(t+\tau)-\phi (t)\|_{L_{2}(\Omega,H)}<\epsilon$ for all $t\in \mathbf{R}$ satisfying the condition $|t-t_{i}|>\epsilon,$ $i\in\mathbf{Z}.$ Let $T(\phi,\epsilon)$ be the set of all $\epsilon$-translation numbers of $\phi.$
\end{definition}

\begin{definition}
A function $\phi\in PC(\mathbf{R},L_{2}(P,H))$ is said to be square mean piecewise almost periodic if the following conditions are fulfilled.
\begin{enumerate}
\item [(1)] $\{t_{i}^{j}=t_{i+j}-t_{i}\}, $ $j\in\mathbf{Z},$ is equipotentialy almost periodic, that is , for any $\epsilon>0,$ there exists a relative dense set $Q_{\epsilon}$ of $\mathbf{R}$ such that for each $\tau\in Q_{\epsilon} $ there is an integer $q\in \mathbf{Z}$ such that $|t_{i+q}-t_{i}-\tau|<\epsilon$ for all $i\in\mathbf{Z}.$
  \item  [(2)] For any $\epsilon>0,$ there exists a positive number $\delta=\delta(\epsilon)$ such that if the points $t$ and $t'$ belong to a same interval of continuity of $\phi$ and $|t-t'|<\delta,$ then $\|\phi(t)-\phi(t')\|_{L_{2}(P,H)}^{2}<\epsilon.$
  \item [(3)] For every $\epsilon>0,$ $T(\phi,\epsilon)$ is a relatively dense set in $\mathbf{R}.$
\end{enumerate}
\end{definition}

We denote by $AP_{T}(\mathbf{R},L_{2}(P,H))$ the collection of all continuous and uniformly bounded  square-mean piecewise almost periodic processes, it  is a Banach space with the norm $\|x\|_{\infty}=\sup_{t\in\mathbf{R}}\|x(t)\|_{L_{2}(P,H)}=\sup_{t\in\mathbf{R}}(\mathbf{E}\|x(t)\|_{H}^{2})^{1/2}.$

\begin{lemma} \cite{15}\label{lem}
Let $f\in AP_{T}(\mathbf{R},L_{2}(P,H)),$ then, $R(f)$, the range of $f$ is a relatively compact set of $L_{2}(P,H).$
\end{lemma}
\begin{definition}
A function $f(t,x,\tilde{x})\in PC(\mathbf{R}\times L_{2}(P,H)\times\mathcal{B},L_{2}(P,H))$ is said to be square-mean piecewise almost periodic in $t \in \mathbf{R}$ and uniform on compact subset of $L_{2}(P,H)\times\mathcal{B}$ if for every $\epsilon>0$ and every compact subset $K\subseteq L_{2}(P,H)\times\mathcal{B},$ there exists a relatively dense subset $\Omega(\epsilon)$ of $R$ such that
\begin{equation*}
\|f(t+\tau,x,\tilde{x})-f(t,x,\tilde{x})\|_{L_{2}(P,H)}<\epsilon,
\end{equation*}
for all $(x,\tilde{x})\in K,\,\tau\in \Omega(\epsilon),\,t\in\mathbf{R}$ satisfying $|t-t_{i}|>\epsilon.$ The collection of all such processes is denoted by $AP_{T}(\mathbf{R}\times L_{2}(P,H)\times\mathcal{B},L_{2}(P,H)).$
\end{definition}

\begin{lemma}
Suppose that $f(t,x,\tilde{x})\in AP_{T}(\mathbf{R}\times L_{2}(P,H)\times\mathcal{B},L_{2}(P,H))$ and $f(t,\cdot,\cdot)$ is uniformly continuous on each compact subset $K\subseteq L_{2}(P,H)$ uniformly for $t\in \mathbf{R}.$ Namely, for all $\epsilon>0,$ there exists $\delta'>0  $ such that when $(x,\tilde{x}),(y,\tilde{y})\in K$ and $\|x-y\|_{L_{2}(P,H)}+\|\tilde{x}-\tilde{y}\|_\mathcal{B}<\delta'$ implies that $\|f(t,x,\tilde{x})-f(t,y,\tilde{y})\|_{L_{2}(P,H)}<\epsilon$ for all $t\in \mathbf{R}.$ Then the function $t\rightarrow f(t,x(t),x_{t})\in AP_{T}(\mathbf{R},L_{2}(P,H))$ for any $x\in AP_{T}(\mathbf{R},L_{2}(P,H)).$
\end{lemma}
\begin{proof}
 Since $x\in AP_{T}(\mathbf{R},L_{2}(P,H))$ by the theorem 1.2.7 of \cite{13}, $t\rightarrow x_{t}\in AP_{T}(\mathbf{R},L_{2}(P,\mathcal{B})).$ By lemma \ref{lem}, $R(x)$ is a relative compact subset of $L_{2}(P,H)$. Because $f(t,\cdot,\cdot)$ is continuous on each compact subset  $K\subseteq L_{2}(P,H)$ uniformly for $t\in \mathbf{R},$ then for any $\epsilon>0,$ there exists a number $\delta': 0<\delta'<\epsilon/4$ such that
 \begin{equation}\label{con}
 \|f(t,x_{1}(t),(x_{1})_{t})-f(t,x_{2}(t),(x_{2})_{t})\|_{L_{2}(P,H)}<\epsilon/4,
 \end{equation}
where $x_{1}(t),x_{2}(t)\in R(x)$ and $\|x_{1}(t)-x_{2}(t)\|_{L_{2}(P,H)}+\|(x_1)_{t}-(x_2)
_{t}\|_\mathcal{B}<\delta'.$
By square mean piecewise almost periodic of $f$, $x(t)$ and $x_{t},$ there exists a relative dense set $\Omega(\epsilon)$ of $\mathbf{R}$ such that the following conditions holds
\begin{equation}\label{con1}
  \|f(t+\tau, x(t_{0}), x_{t_{0}})-f(t, x(t_{0}), x_{t_{0}})\|_{L_{2}(P,H)}<\epsilon/4,
\end{equation}
\begin{equation}\label{con2}
\|x(t+\tau)-x(t)\|_{L_{2}(P,H)}<\epsilon/4,\quad\|x_{t+\tau}-x_{t}\|_{\mathcal{B}}<\epsilon/4,
\end{equation}
for every $x(t_{0})\in R(x),x_{t_{0}}=x(t_{0}+s),\,s\in(-\infty,0],\,t\in\mathbf{R},\,|t-t_{i}|>\epsilon,\,i\in\mathbf{Z},\,\tau\in\Omega(\epsilon).$ Note that
\begin{eqnarray}\label{coc}
 \nonumber  &&\|f(t+\tau, x(t+\tau), x_{t+\tau})-f(t, x(t), x_{t})\|_{L_{2}(P,H)}\\
\nonumber  &\leq &\|f(t+\tau, x(t+\tau), x_{t+\tau})-f(t+\tau, x(t), x_{t})\|_{L_{2}(P,H)}\\\nonumber &&+\|f(t+\tau, x(t), x_{t})-f(t, x(t), x_{t})\|_{L_{2}(P,H)}.
\end{eqnarray}
Combining (\ref{con}),(\ref{con1}) and (\ref{con2}) it follows that
\begin{equation*}
  \|f(t+\tau, x(t+\tau), x_{t+\tau})-f(t, x(t), x_{t})\|_{L_{2}(P,H)}<\epsilon.
\end{equation*}
\end{proof}

Obviously the uniform continuity is weaker than the Lipschitz condition. Then we get the following Corollary.
\begin{corollary}
Suppose that $f(t,x,\tilde{x})\in AP_{T}(\mathbf{R}\times L_{2}(P,H)\times\mathcal{B},L_{2}(P,H))$ and there exists a positive number $M_{f}$ such that for all $t\in \mathbf{R}$ and $(x_{1},\tilde{x_{1}}),\,(x_{2},\tilde{x_{2}})\in L_{2}(P,H)\times\mathcal{B}$
\begin{eqnarray}\label{lipf}
  \|f(t,x_{1},\tilde{x_{1}})-f(t,x_{2},\tilde{x_{2}})\|_{L_{2}(P,H)}&\leq&M_{f}(\|x_{1}-x_{2}\|_{L_{2}(P,H)}+\|\tilde{x_{1}}-\tilde{x_{2}}\|_{\mathcal{B}}).
\end{eqnarray}
Then the function $t\rightarrow f(t,x(t),x_{t})\in AP_{T}(\mathbf{R},L_{2}(P,H))$ for any $x\in AP_{T}(\mathbf{R},L_{2}(P,H)).$
\end{corollary}
\begin{lemma}\cite{14}\label{lem3} Assume that $f\in AP_{T}(\mathbf{R},L_{2}(P,H)),$ the sequence $\{x_{i}:i\in \mathbf{Z}\}$ is almost periodic in $L_{2}(P,H),$ and $\{t_{i}^{j}=t_{i+j}-t_{i}\},\,i\in\mathbf{Z},\,j=0,\,\pm1,\,\pm2,\ldots,$ are equipotentially almost periodic. Then for each $\epsilon>0,$ there are relative dense sets $\Omega_{\epsilon,f,x_{i}}$ of $\mathbf{R}$ and $Q_{\epsilon,f,x_{i}}$ of $\mathbf{Z}$ such that the following conditions hold.\\
(i) $\|f(t+\tau)-f(t)\|_{L_{2}(P,H)}<\epsilon$ for all $t\in\mathbf{R},$ $|t-t_{j}|>\epsilon,$ $\tau\in \Omega_{\epsilon,f,x_{i}},$ and $i\in\mathbf{Z}.$\\
(ii) $\|x_{i+q}-x_{i}\|_{L_{2}(P,H)}<\epsilon$ for all $q\in Q_{\epsilon,f,x_{i}}$ and $i\in \mathbf{Z}.$\\
(iii) For every $\tau\in \Omega_{\epsilon,f,x_{i}},$ there exists at least one number $q\in Q_{\epsilon,f,x_{i}}$ such that $|t_{i}^{q}-\tau|<\epsilon,\,i\in\mathbf{Z}.$
\end{lemma}
To prove our results, we need the following notations. Let $h:\mathbf{R}\rightarrow \mathbf{R}$ be a continuous function such that $h(t)\geq 1$ for all $t\in \mathbf{R}$ and $h(t)\rightarrow \infty$ as $t\rightarrow\infty.$ We consider the space
\begin{equation}
  (PC)_{h}^{0}(\mathbf{R},L_{2}(P,H))=\{u\in PC(\mathbf{R},L_{2}(P,H)):\,\lim_{|t|\rightarrow\infty}\frac{\|u(t)\|_{L_{2}(P,H)}}{h(t)}=0\}.
\end{equation}
which endowed with the norm $\|u\|_{h}=\sup_{t\in\mathbf{R}}\frac{\|u(t)\|_{L_{2}(P,H)}}{h(t)},$ is a Banach space.
\begin{lemma}\label{lem2}\cite{9}
A set $B\subseteq (PC)_{h}^{0}(\mathbf{R},L_{2}(P,H))$ is a relatively compact set if and only if\\
(1) $\lim_{|t|\rightarrow\infty}\frac{\|u(t)\|_{L_{2}(P,H)}}{h(t)}=0$ uniformly for $x\in B.$\\
(2)$B(t)=\{x(t):x\in B\}$ is relatively compact in $L_{2}(P,H)$ for every $t\in \mathbf{R}.$\\
(3)The set $B$ is equicontinuous on each interval $(t_{i},t_{i+1})\,(i\in \mathbf{Z}).$
\end{lemma}
The following  Krasnoselskii's fixed point theorem appearing in \cite{14}.
\begin{theorem}
Let $\mathcal{M}$ be a closed convex nonempty subset of a Banach space
$(X, \|\cdot\|)$. Suppose that $A,\, B:\,\mathcal{M}\rightarrow X$, such that
\\
(i) $Ax+By\in \mathcal{M}\,(\forall x,y\in\mathcal{M}),$\\
(ii) $A$ is completely continuous,\\
(iii) $B$ a contraction with contraction constant $k < 1.$
Then there is a $y\in\mathcal{M}$ with $Ay+By=y.$
\end{theorem}

\section{Existence result}
In this section, we aim to give some sufficient conditions which guarantee the existence and the uniqueness of the almost periodic mild solution for system (\ref{eq1}).
\begin{definition}
An $\mathcal{F}_{t}$-adapted $H$- valued stochastic process $x(t)$ defined on $\mathbf{R}$ is called the mild solution for (\ref{eq1}) if
\begin{enumerate}
  \item [(a)] $\{x_{t}:\,t\in \mathbf{R}\}$ is $\mathcal{B}$ valued and $x(\cdot)\in PC(\mathbf{R},L_{2}(P,H))$;
  \item [(b)] $\int_{-\infty}^{t}\|x(u)\|_{H}^{2}\mathrm{d}u<\infty$ almost surely;
  \item [(c)] for any $t\in(t_{i},t_{i+1}]$, $x(t)$ satisfies the following integral equation:
  \begin{eqnarray}\label{sol}
\left\{ \begin{array}{ll}
 x(t)= T(t-\sigma)(\varphi(0)-G(\sigma,\varphi(0),\varphi))+G(t,x(t),x_{t})+\int_{\sigma}^{t}A T(t-s)G(s,x(s),x_{s})\mathrm{d}s\\
\quad\quad\quad+\int_{\sigma}^{t}T(t-s)f(s,x(s),x_{s})\mathrm{d}s+\int_{\sigma}^{t}T(t-s)g(s,x(s),x_{s})\mathrm{d}w(s)\\
  \quad\quad\quad+\sum_{\sigma<t_{i}<t}T(t-t_{i})I_{i}(x(t_{i})),\\
x_{\sigma}=\varphi\in\mathcal{B}.
\end{array} \right.
\end{eqnarray}
\end{enumerate}
\end{definition}

In order to get the existence of square mean piecewise almost periodic solution of system (\ref{eq1}), we introduce the following assumptions.
\begin{enumerate}
  \item[(A1)] The operator $A: D(A)\subseteq H\rightarrow H$ is the infinitesimal generator of an exponentially stable $C_{0}$-semigroup $\{T(t):t\geq0\}$ on $L_{2}(P,H)$, i.e., $\|T(t)\|\leq M e^{-\delta t}, t\geq 0, M, \delta>0.$ Moreover, $T(t)$ is compact for $t>0,$
     and $0\in \rho(A),$ where $\rho(A)$ is the resolvent set of $A.$ For $\alpha\in (0,1],$ it is possible to define the fractional power $(-A)^{\alpha}$ as a closed linear operator on its domain $D((-A)^{\alpha})$ and $D((-A)^{\alpha})$ is dense in $H.$ The expression $\|h\|_{\alpha}=\|(-A)^{\alpha}h\|_{H},\, h\in D((-A)^{\alpha}),$ defines a norm in $D((-A)^{\alpha}).$ Let $H_{\alpha}$ represent the space $D((-A)^{\alpha})$ endowed with the norm $\|\cdot\|_{\alpha},$ then for every $0<\alpha\leq 1$ $H_{\alpha}$ is a Banach space. There exists $M_{\alpha} >0$ such that $\|(-A)^{\alpha}T(t)\|\leq M_{\alpha}t^{-\alpha}e^{-\delta t},\,t>0,\,\delta>0.$

  \item[(A2)] $f\in\mathcal{A}\mathcal{P}_{T}(\mathbf{R}\times L_{2}(P,H)\times\mathcal{B},L_{2}(P,H)),\,g(t,\cdot,\cdot)\in\mathcal{A}\mathcal{P}_{T}\big(\mathbf{R}\times L_{2}(P,H)\times\mathcal{B},\mathcal{L}_{2}^{0}\big)$ for each compact set $K\subseteq L^{2}(P,H)\times\mathcal{B},$  $f(t,\cdot,\cdot),\,g(t,\cdot,\cdot)$ and $G(t,\cdot,\cdot),$ are uniformly continuous in each compact set $K\subseteq L^{2}(P,H)\times\mathcal{B}$ uniformly for $t\in \mathbf{R}.$ $I_{i}(x)\,(i\in \mathbf{Z})$ is almost periodic uniformly in $x\in H$ uniformly continuous function defined on $K.$ Let
 \begin{eqnarray*}
  &&F_{L}=\sup_{\{t\in \mathbf{R},\, \max\{\|x\|_{L_{2}(P,H)},\|\tilde{x}\|_{\mathcal{B}}\}<L\}}\|f(t,x,\tilde{x})\|_{L_{2}(P,H)}<\infty,\\
 && G_{L}=\sup_{\{t\in \mathbf{R},\, \max\{\|x\|_{L_{2}(P,H)},\|\tilde{x}\|_{\mathcal{B}}\}<L\} }\|g(t,x,\tilde{x})\|_{L_{2}(P,\mathcal{L}_{2}^{0})}<\infty,\\ &&I_{L}=\sup_{\{i\in \mathbf{Z}, \, \|x\|_{L_{2}(P,H)}<L\} }\|I(x)\|_{L^{2}(P,H)}<\infty,
\end{eqnarray*}
 where $L $ is an arbitrary positive number.
 \item[(A3)]  There are two constants $\alpha\in(0,1),$  $M_{G}>0$ such that the function $G\in AP(\mathbf{R}\times L_{2}(P,H)\times\mathcal{B},L_{2}(P,H_{\alpha})),$ and for any $(x_{1},\tilde{y}_{1}),\,(x_2,\tilde{y}_{2})\in L_{2}(P,H)\times\mathcal{B}$, $t\in\mathbf{R},$
     \begin{equation*}
    \|(-A)^{\alpha}G(t,x_{1},\tilde{y}_{1})-(-A)^{\alpha}G(t,x_{2},\tilde{y}_{2})\|_{L_{2}(P,H)}\leq M_{G}(\|x_{1}-x_{2}\|_{L_{2}(P,H)}+\|\tilde{y}_{1}-\tilde{y}_{2}\|_{\mathcal{B}}).
     \end{equation*}
    We further assume that $G(t,0,0)\equiv0$ for all $t\in \mathbf{R}.$
   \item[(A4)] Let $\{x_{n}\}\subseteq \mathcal{AP}_{T}({\mathbf{R},L_{2}(P,H)})$ be uniformly bounded in $R$ and uniformly convergent in each compact set of $\mathbf{R},$ then $f(t, x_{n}(t),(x_{n})_{t})$ is relatively compact in $PC(\mathbf{R},L_{2}(P,H)).$
\end{enumerate}

Since $\|T(t-\sigma)\|\leq M e^{-\delta (t-\sigma)},$  for all $t\geq \sigma,$ let $\sigma\rightarrow -\infty$ then we have $\|T(t-\sigma)\|\rightarrow 0,$ and the above (\ref{sol}) can be replaced by
\begin{eqnarray}\label{sol2}
\left\{ \begin{array}{ll}
 x(t)= G(t,x(t),x_{t})+\int_{-\infty}^{t}A T(t-s)G(s,x(s),x_{s})\mathrm{d}s\\
\quad\quad\quad+\int_{-\infty}^{t}T(t-s)f(s,x(s),x_{s})\mathrm{d}s+\int_{-\infty}^{t}T(t-s)g(s,x(s),x_{s})\mathrm{d}w(s)\\
  \quad\quad\quad+\sum_{t_{i}<t}T(t-t_{i})I_{i}(x(t_{i})).
\end{array} \right.
\end{eqnarray}
\begin{theorem}\label{theorem}
Let (A1)-(A4) be satisfied, $( \|(-A)^{-\alpha}\|+\frac{\Gamma(\alpha)}{\delta} M_{1-\alpha})M_{G}<\frac{1}{8}$ and there is a positive number $L_{0}$ such that $\frac{M}{\delta}  F_{L_{0}}+\frac{M}{2\delta}G_{L_{0}}+\frac{M}{1-e^{-\delta r}}I_{L_{0}}\leq \frac{L_{0}}{2\sqrt{6}},$  then (\ref{eq1}) has a unique square mean almost periodic mild solution.
\end{theorem}
\begin{proof}
%
%

Let $B=\{x\in \mathcal{AP}_{T}( \mathbf{R},L_{2}(P,H)):\|x\|_{L_{2}(P,H)}\leq L_{0}\}.$
We define the operator $\Phi$ on $\mathcal{AP}_{T}( \mathbf{R},L_{2}(P,H))$ by
\begin{eqnarray}\label{phi}
\nonumber\Phi x(t)&=&G(t,x(t),x_{t})+\int_{-\infty}^{t}A T(t-s)G(s,x(s),x_{s})\mathrm{d}s\\
&&+\int_{-\infty}^{t}T(t-s)f(s,x(s),x_{s})\mathrm{d}s+\int_{-\infty}^{t}T(t-s)g(s,x(s),x_{s})\mathrm{d}w(s)\\
  \nonumber &&+\sum_{t_{i}<t}T(t-t_{i})I_{i}(x(t_{i})), t\in \mathbf{R}.
\end{eqnarray}
We aim to show that the operator $\Phi$ has a fixed point on $B,$ which implies (\ref{eq1}) has a unique square mean almost periodic mild solution. To this end, we decompose $\Phi$ as $\Phi=\Phi_{1}+\Phi_{2},$ where $\Phi_{1},$ $\Phi_{2}$ are defined on $B$, respectively, by
\begin{eqnarray}
\nonumber\Phi_{1}x(t)&=&G(t,x(t),x_{t})+\int_{-\infty}^{t}A T(t-s)G(s,x(s),x_{s})\mathrm{d}s,\\
    \nonumber \Phi_{2}x(t)&=&\int_{-\infty}^{t}T(t-s)f(s,x(s),x_{s})\mathrm{d}s\\
    \nonumber &&+\int_{-\infty}^{t}T(t-s)g(s,x(s),x_{s})\mathrm{d}w(s)+\sum_{t_{i}<t}T(t-t_{i})I_{i}(x(t_{i})).
\end{eqnarray}
Our proof will be split into the following three steps.

Step 1. In what following, we prove that for any $x,y\in B,$  $\Phi_{1} x+\Phi_{2}y\in B.$
\begin{eqnarray*}
 &&\|\Phi_{1} x\|_{L_{2}(P,H)}\\
  &=& \|G(t,x(t),x_{t})+\int_{-\infty}^{t}A T(t-s)G(s,x(s),x_{s})\mathrm{d}s\|_{L_{2}(P,H)} \\
 &\leq&\sqrt{2} \|G(t,x(t),x_{t})\|_{L_{2}(P,H)}+\sqrt{2}  \|\int_{-\infty}^{t}A T(t-s)G(s,x(s),x_{s})\mathrm{d}s\|_{L_{2}(P,H)}\\
 &\leq& \sqrt{2} \|(-A)^{-\alpha}\|\|(-A)^{\alpha}G(t,x(t),x_{t})\|_{L_{2}(P,H)}\\&&+\sqrt{2} \|\int_{-\infty}^{t}(-A)^{1-\alpha}T(t-s)[(-A)^{\alpha}G(s,x(s),x_{s})]\mathrm{d}s\|_{L_{2}(P,H)}\\
 &\leq&2\sqrt{2}  \|(-A)^{-\alpha}\|M_{G}\sup_{s\leq t}\|x(s)\|_{L_{2}(P,H)}+2\sqrt{2} \frac{\Gamma(\alpha)}{\delta} M_{1-\alpha}M_{G}\sup_{s\leq t}\|x(s)\|_{L_{2}(P,H)}\\
 &=&2\sqrt{2} ( \|(-A)^{-\alpha}\|+\frac{\Gamma(\alpha)}{\delta} M_{1-\alpha})M_{G}\|x\|_{\infty}.
 \end{eqnarray*}
\begin{eqnarray*}
  \|\Phi_{2} y\|_{L_{2}(P,H)}
   &\leq& \sqrt{3}[\mathbf{E}(\int_{-\infty}^{t}\|T(t-s)\|\|f(s,y(s),y_{s})\|_{H}\mathrm{d}s)^{2} ]^{\frac{1}{2}}\\ &&+\sqrt{3}(\mathbf{E}\int_{-\infty}^{t}\|T(t-s)\|^{2}\|g(s,y(s),y_{s})\|^{2}_{\mathcal{L}^{0}_{2}}\mathrm{d}s)^{\frac{1}{2}}\\&&
   +\sqrt{3}\|\sum_{t_{i}<t}T(t-t_{i})I_{i}(y(t_{i}))\|_{_{L_{2}(P,H)}}\\
   &\leq& \sqrt{3}(\frac{M}{\sigma}  F_{L_{0}}+\frac{M}{2\sigma}G_{L_{0}}+M\sum_{0\leq k=j-i<\infty}e^{-\delta k r}I_{L_{0}}) \\
   &=&  \sqrt{3}( \frac{M}{\delta}  F_{L_{0}}+\frac{M}{2\delta}G_{L_{0}}+\frac{M}{1-e^{-\delta r}}I_{L_{0}}).
\end{eqnarray*}
Then by using the elementary inequality $|a+b|^{2}\leq 2(a^{2}+b^{2}),$ we can show that
\begin{eqnarray}\label{condi}
\|\Phi_{1} x+\Phi_{2}y\|_{L_{2}(P,H)}&\leq&\sqrt{2}\|\Phi_{1} x\|_{L_{2}(P,H)}+\sqrt{2}\|\Phi_{2}y\|_{L_{2}(P,H)}\\
\nonumber&\leq&4( \|(-A)^{-\alpha}\|+\frac{\Gamma(\alpha)}{\delta} M_{1-\alpha})M_{G}\|x\|_{\infty}\\\nonumber&&+\sqrt{6}(\frac{M}{\delta}  F_{L_{0}}+\frac{M}{2\delta}G_{L_{0}}+\frac{M}{1-e^{-\delta r}}I_{L_{0}}).
\end{eqnarray}
Due to
\begin{eqnarray}\label{conditon1}
  ( \|(-A)^{-\alpha}\|+\frac{\Gamma(\alpha)}{\delta} M_{1-\alpha})M_{G}\leq \frac{1}{8}\, \text{and}\,\frac{M}{\delta}  F_{L_{0}}+\frac{M}{2\delta}G_{L_{0}}+\frac{M}{1-e^{-\delta r}}I_{L_{0}}\leq \frac{L_{0}}{2\sqrt{6}},
\end{eqnarray}
then substituting (\ref{conditon1}) into (\ref{condi}) yields that
\begin{equation}\label{0}
\|\Phi_{1} x+\Phi_{2}y\|_{L_{2}(P,H)}\leq L_{0}.
\end{equation}

By (A2) and Lemma \ref{lem}, we know that $G(t,x(t),x_{t})$ and $f(t,x(t),x_{t})\in \mathcal{AP}_{T}(R,L_{2}(P,H)),$  $g(t,x(t),x_{t})\in \mathcal{AP}_{T}(R,L_{2}(P,\mathcal{L}^{0}_{2})),$ $\{I_{i}(x(t_{i}))\}$ is almost periodic. According to Lemma \ref{lem3}, for every $\epsilon>0,$ there exist relatively dense sets $\Omega_{\epsilon,G,f,g,I_{i}}$ of $R$ and $Q_{\epsilon,G,f,g,I_{i}}$ of $\mathbf{Z}$ such that for $\tau\in \Omega_{\epsilon,G,f,g,I_{i}},$ $\exists$ $q\in Q_{\epsilon,G,f,g,I_{i}},$ such that $\|x(t+\tau)-x(t)\|_{L_{2}(P,H)}<\epsilon,$
\begin{eqnarray}
\nonumber && \|G(t+\tau,x(t+\tau),x_{t+\tau})-G(t,x(t),x_{t})\|_{L_{2}(P,H))}\leq \epsilon,\, \\\nonumber&& \|f(t+\tau,x(t+\tau),x_{t+\tau})-f(t,x(t),x_{t})\|_{L_{2}(P,H))}\leq \epsilon,\\
&&\nonumber  \|g(t+\tau,x(t+\tau),x_{t+\tau})-g(t,x(t),x_{t})\|_{L_{2}(P,\mathcal{L}^{0}_{2})}\leq \epsilon,\, |t_{i}^{q}-\tau|\leq \epsilon,
 \end{eqnarray}
 where $t\in \mathbf{R},$ $|t-t_{i}|>\epsilon,$ $i\in \mathbf{Z}.$ Then for $\tau\in\Omega_{\epsilon,G,f,g,I_{i}},$ we have
 \begin{eqnarray}\label{1}
 \nonumber&&\|(-A)^{\alpha}G(s+\tau,x(s+\tau),x_{s+\tau})-(-A)^{\alpha}G(s,x(s),x_{s})\|_{L_{2}(P,H)}<2M_{G}\epsilon,\\
 \nonumber&& \|\Phi_{1}x(t+\tau)-\Phi_{1}x(t)\|_{L_{2}(P,H))} \\
  \nonumber&\leq& \sqrt{2}\|G(t+\tau,x(t+\tau),x_{t+\tau})-G(t,x(t),x_{t})\|_{L_{2}(P,H)}\\\nonumber&&+\sqrt{2} \|\int_{-\infty}^{t}A T(t-s)[G(s+\tau,x(s+\tau),x_{s+\tau})-G(s,x(s),x_{s})]\mathrm{d}s\|_{L_{2}(P,H)}  \\
    \nonumber&\leq& \sqrt{2} \epsilon+\sqrt{2}[\mathbf{E}(\int_{-\infty}^{t}(-A)^{1-\alpha}T(t-s)\|(-A)^{\alpha}G(s+\tau,x(s+\tau),x_{s+\tau})-(-A)^{\alpha}G(s,x(s),x_{s})\|_{H}\mathrm{d}s)^{2}]^{\frac{1}{2}}\\
    \nonumber&\leq& \sqrt{2}\epsilon+\sqrt{2}M_{1-\alpha}[\int_{-\infty}^{t}e^{-\delta(t-s)}(t-s)^{\alpha-1}\mathrm{d}s\mathbf{E}\int_{-\infty}^{t}e^{-\delta(t-s)}(t-s)^{\alpha-1}\\&&\|(-A)^{\alpha}G(s+\tau,x(s+\tau),x_{s+\tau})-(-A)^{\alpha}G(s,x(s),x_{s})\|_{H}^{2}\mathrm{d}s]^{\frac{1}{2}}\\
     \nonumber&\leq&\sqrt{2}\epsilon+2\sqrt{2}M_{1-\alpha}\frac{\Gamma(\alpha)}{\delta}M_{G}\epsilon,
 \end{eqnarray}
 \begin{eqnarray}\label{2}
 \nonumber&& \|\Phi_{2}y(t+\tau)-\Phi_{2}y(t)\|_{L_{2}(P,H))} \\
 \nonumber&\leq&\sqrt{3} [\mathbf{E}(\int_{-\infty}^{t}\|T(t-s)\|\|f(s+\tau,y(s+\tau),y_{s+\tau})-f(s,y(s),y_{s})\|_{H}\mathrm{d}s)^{2}]^{\frac{1}{2}} \\\nonumber&& +\sqrt{3}(\mathbf{E}\int_{-\infty}^{t}\|T(t-s)\|^{2}\|g(s+\tau,y(s+\tau),y_{s+\tau})-g(s,y(s),y_{s})\|^{2}_{\mathcal{L}_{2}^{0}}\mathrm{d}s)^{\frac{1}{2}}\\
   \nonumber&&+\sqrt{3}\|\sum_{t_{i}<t}T(t-t_{i})[I_{i+q}(y(t_{i+q}))-I_{i}(y(t_{i}))]\|_{L_{2}(P,H)}\\
 &\leq&\sqrt{3} \frac{M}{\delta}\epsilon+\sqrt{3}\frac{M}{\sqrt{2}\delta}\epsilon+\sqrt{3}\frac{M}{1-e^{-\delta r}}\epsilon.
  \end{eqnarray}
  Combining (\ref{0}) (\ref{1}) and (\ref{2}), it follows that for any $x,y\in B,$  $\Phi_{1} x+\Phi_{2}y\in B.$

Step 2.
 We show that $\Phi_{1}$ is a contraction.

Let $x,\,y\in B,$ then for each $t\in \mathbf{R},$ we have
\begin{eqnarray}\label{esti1}
\nonumber&&\|(\Phi_{1}x)(t)-(\Phi_{1}y)(t)\|_{L_{2}(P,H)}\\
\nonumber&\leq&\sqrt{2}\|G(t,x(t),x_{t})-G(t,y(t),y_{t}) \|_{L_{2}(P,H)}\\
\nonumber&&\quad+\sqrt{2}\|\int_{-\infty}^{t}A T(t-s)G(s,x(s),x_{s})-AT(t-s)G(s,y(s),y_{s})\mathrm{d}s\|_{L_{2}(P,H)}\\
\nonumber&\leq &\sqrt{2}\|(-A)^{-\alpha}\|\|(-A)^{\alpha}[G(t,x(t),x_{t})-G(t,y(t),y_{t})]\|_{L_{2}(P,H)}\\
\nonumber&&\quad+\sqrt{2}\|M_{G}M_{1-\alpha}(\int_{-\infty}^{t}e^{-\delta(t-s)}(t-s)^{\alpha-1}\mathrm{d}s)\sup_{s\leq t}(\|x(s)-y(s)\|_{L_{2}(P,H)}+\|x_{s}-y_{s}\|_{\mathcal{B}})\\
\nonumber&\leq&2\sqrt{2}[ \|(-A)^{-\alpha}\|+\frac{\Gamma(\alpha)}{\delta} M_{1-\alpha}]M_{G}\|x-y\|_{\infty}.
\end{eqnarray}
By the condition $[\|(-A)^{-\alpha}\|+\frac{\Gamma(\alpha)}{\delta} M_{1-\alpha}]M_{G}<\frac{1}{8}$, we see $\Phi_{1}$ is a contraction on $ B.$

 Let $\{x_{n}\}\subseteq \mathcal{AP}_{T}(\mathbf{R},L_{2}(P,H)),\,x_{n}\rightarrow x$ in $\mathcal{AP}_{T}(\mathbf{R},L_{2}(P,H))$ as $n\rightarrow\infty;$ by Lemma\ref{lem} , there is a compact subset $B_{0}\subseteq L_{2}(P,H)$ such that $x_{n}(t),\,x(t)\in B_{0}$ for all $t\in \mathbf{R},\,n\in\mathbf{N};$  here we assume $B\subseteq B_{0}.$ By (A2) for any given $\epsilon,$ there exists $\delta'>0$  such that when $(x,\tilde{x}),(y,\tilde{y})\in B_{0}$ and $\|x-y\|_{L_{2}(P,H)}+\|\tilde{x}-\tilde{y}\|_\mathcal{B}<\delta'$ implies that \\ $\|f(t,x,\tilde{x})-f(t,y,\tilde{y})\|_{L_{2}(P,H)}<\epsilon,$ $\|g(t,x,\tilde{x})-g(t,y,\tilde{y})\|_{L_{2}(P,\mathcal{L}^{0}_{2})}<\epsilon,$ $\|I(x(t_{i}))-I(y(t_{i}))\|_{L_{2}(P,H)}<\epsilon.$ For the above $\delta',$ there exists $n_{0}$ such that $\|x_{n}(t)-x(t)\|_{L_{2}(P,H)}+\|(x_{n})_{t}-x_{t}\|_{\mathcal{B}}<\frac{\delta'}{L_{0}}$ for $n>n_{0}$ and $t\in\mathbf{R},$ then

 $\|f(t,x_{n}(t),(x_{n})_{t})-f(t,x(t),x_{t})\|_{L_{2}(P,H)}<\epsilon,\,\|g(t,x_{n}(t),(x_{n})_{t})-g(t,x(t),x_{t})\|_{\mathcal{L}_{2}^{0}}<\epsilon,\\
  \|I(x_{n}(t_{i}))-I(x(t_{i}))\|_{L_{2}(P,H)}<\epsilon,$ for $n>n_{0}$ and $t\in\mathbf{R}.$ Hence
\begin{eqnarray*}
 &&\|\Phi_{2}(x_{n})(t)-\Phi_{2}(x)(t)\|_{L_{2}(P,H)}  \\
   &\leq&\sqrt{3} [\mathbf{E}(\int_{-\infty}^{t}\|T(t-s)\|\|f(t,x_{n}(t),(x_{n})_{t})-f(t,x(t),x_{t})\|_{H}\mathrm{d}s)^{2}]^{\frac{1}{2}} \\ && +\sqrt{3}[\mathbf{E}\int_{-\infty}^{t}\|T(t-s)\|^{2}\|g(t,x_{n}(t),(x_{n})_{t})-g(t,x(t),x_{t})\|_{\mathcal{L}_{2}^{0}}^{2}\mathrm{d}s]^{\frac{1}{2}}\\
   &&+\sqrt{3}[\mathbf{E}(\sum_{t_{i}<t}\|T(t-t_{i})\|\|I(x_{n}(t_{i}))-I(x(t_{i}))\|_{H})^{2}]^{\frac{1}{2}}\\
  &\leq&  \sqrt{3} M\epsilon\int_{-\infty}^{t} e^{-\delta(t-s)}\mathrm{d}s+\sqrt{3}[\int_{-\infty}^{t}M^{2} e^{-2\delta(t-s)}\epsilon^{2}\mathrm{d}s]^{\frac{1}{2}}+\sqrt{3}\sum_{t_{i}<t}M e^{-\delta(t-t_{i})}\epsilon\\
  &\leq& \sqrt{3}(\frac{M}{\delta}\epsilon+\frac{M}{\sqrt{2}\delta}\epsilon+\frac{M}{1-e^{\delta r}} \epsilon),
\end{eqnarray*}
 for all $n>n_{0}$ $t\in \mathbf{R}.$ This implies $\Phi_{2}$ is continuous.

 Step3. $\Phi_{2}$ is completely continuous on $ B.$

Claim 1. The set of functions $\Phi_{2}(B)$ is equicontinuous at each interval $(t_{i},t_{i+1})$ $(i\in \mathbf{Z}).$
Let $\epsilon>0$ small enough and $t_{i}<t^{'}<t^{''}<t_{i+1},$ $i\in \mathbf{Z},$ we get
\begin{eqnarray*}
  (\Phi_{2}x)(t'')-(\Phi_{2}x)(t')&=& \int_{t'}^{t''}T(t''-s)f(s,x(s),x_{s})\mathrm{d}s
     +\int_{t'}^{t''}T(t''-s)g(s,x(s),x_{s})\mathrm{d}w(s)\\&&+\int_{-\infty}^{t'}[T(t''-s)-T(t'-s)]f(s,x(s),x_{s})\mathrm{d}s\\&&
    +\int_{-\infty}^{t'}[T(t''-s)-T(t'-s)]g(s,x(s),x_{s})\mathrm{d}w(s)\\
    &&+\sum_{t_{i}<t'}[T(t''-t_{i})-T(t'-t_{i})]I_{i}(x(t_{i})).
  \end{eqnarray*}
Since $\{T(t):t\geq 0\}$ is a $C_{0}$-semigroup, there exists $\mu<\min\{r,\,\frac{\epsilon}{5\sqrt{5}MF_{L_{0}}},\,\frac{\epsilon}{5\sqrt{5}MG_{L_{0}}}\}$ such that $t''-t'<\mu$ implies that $\|T(t''-t')-I\|<\min\{\frac{ \epsilon \delta}{5\sqrt{5}MF_{L_{0}}},\frac{ \epsilon \delta}{5\sqrt{5}MG_{L_{0}}},\frac{(1-e^{-\delta r})\epsilon}{5\sqrt{5}MI_{0}}\}.$
An application of the Cauchy-Schwarz inequality, we get
\begin{eqnarray*}
&&\| \int_{t'}^{t''}T(t''-s)f(s,x(s),x_{s})\mathrm{d}s\|_{L_{2}(P,H)}\\&\leq&[\mathbf{E}( \int_{t'}^{t''}Me^{-\delta(t''-s)}\|f(s,x(s),x_{s})\|_{H}\mathrm{d}s)^{2}]^{\frac{1}{2}}\\
 &\leq&M(\int_{t'}^{t''}e^{-\delta(t''-s)}\mathrm{d}s\mathbf{E}\int_{t'}^{t''}e^{-\delta(t''-s)}\|f(s,x(s),x_{s})\|_{H}^{2}\mathrm{d}s)^{\frac{1}{2}}\\
  &\leq& MF_{L_{0}}(t'-t'')\\
  &\leq& MF_{L_{0}}\frac{ \epsilon}{5\sqrt{5}MF_{L_{0}}}=\frac{\epsilon}{5\sqrt{5}}.
\end{eqnarray*}
\begin{eqnarray*}
  &&[\mathbf{E}\|\int_{t'}^{t''}T(t''-s)g(s,x(s),x_{s})\mathrm{d}w(s)\|_{H}^{2}]^{\frac{1}{2}}\\ &=&[ \mathbf{E}\int_{t'}^{t''}  \|T(t''-s)g(s,x(s),x_{s})\|_{\mathcal{L}_{0}^{2}}\mathrm{d}s]^{\frac{1}{2}}\\
   &\leq& [\mathbf{E}\int_{t'}^{t''} M^{2}e^{-2\delta(t''-s)}\|g(s,x(s),x_{s})\|_{H}^{2}\mathrm{d}s]^{\frac{1}{2}} \\
   &\leq& MG_{L_{0}}(t'-t'')\\
    &\leq& MG_{L_{0}}\frac{ \epsilon}{5\sqrt{5}MG_{L_{0}}}=\frac{\epsilon}{5\sqrt{5}}.
\end{eqnarray*}
\begin{eqnarray*}
&& \|\int_{-\infty}^{t'}[T(t''-s)-T(t'-s)]f(s,x(s),x_{s})\mathrm{d}s \|_{L_{2}(P,H)}\\&=&  \|\int_{-\infty}^{t'}[T(t''-t)-I]T(t'-s)f(s,x(s),x_{s})\mathrm{d}s \|_{L_{2}(P,H)}\\
  &\leq&[\mathbf{E}(\int_{-\infty}^{t'}\|T(t''-t')-I\|Me^{-\delta(t'-s)}\|f(s,x(s),x_{s})\|_{H}\mathrm{d}s )^{2}]^{\frac{1}{2}} \\
   &\leq& M\|T(t''-t')-I\|[\int_{-\infty}^{t'}e^{-\delta(t'-s)}\mathrm{d}s\mathbf{E}(\int_{-\infty}^{t'}e^{-\delta(t'-s)}\|f(s,x(s),x_{s})\|_{H}^{2}\mathrm{d}s)]^{\frac{1}{2}}\\
     &\leq& \frac{ \epsilon \delta}{5\sqrt{5}MF_{L_{0}}}\frac{M}{\delta}F_{L_{0}}=\frac{\epsilon}{5\sqrt{5}}.
\end{eqnarray*}
\begin{eqnarray*}
 && [\mathbf{E}\|\int_{-\infty}^{t'}[T(t''-s)-T(t'-s)]g(s,x(s),x_{s})\mathrm{d}w(s)\|^{2}]^{\frac{1}{2}}\\&\leq&[\mathbf{E}\int_{-\infty}^{t'}\|[T(t''-t')-I]T(t'-s)g(s,x(s),x_{s})\|^{2}\mathrm{d}s]^{\frac{1}{2}} \\
  &\leq& [\mathbf{E}\int_{-\infty}^{t'}\|T(t''-t')-I\|^{2}M^{2}e^{-2\delta(t'-s)}\|g(s,x(s),x_{s})\|_{\mathcal{L_{2}^{0}}}^{2}\mathrm{d}s]^{\frac{1}{2}}\\
   &\leq& [\|T(t''-t)-I\|^{2} \int_{-\infty}^{t'}M^{2}e^{-2\delta(t'-s)}G_{L_{0}}\mathrm{d}s]^{\frac{1}{2}}\\
   &\leq& \frac{ \epsilon \delta}{5\sqrt{5}MG_{L_{0}}}\frac{M}{\sqrt{2}\delta}G_{L_{0}}<\frac{\epsilon}{5\sqrt{5}}.
\end{eqnarray*}
\begin{eqnarray*}
  &&\|\sum_{t_{i}<t'}[T(t''-t_{i})-T(t'-t_{i})]I_{i}(x(t_{i}))\|_{L(P,H)}\\
  &\leq&[\mathbf{E}(\sum_{t_{i}<t'}\|T(t''-t')-I\|\|T(t'-t_{i})\|\|I_{i}(x(t_{i}))\|_{H})^{2}]^{\frac{1}{2}}\\
  &\leq&\sum_{t_{i}<t'}\frac{(1-e^{-\delta r})\epsilon}{5MI_{0}}M e^{-\delta (t'-t_{i})}I_{0}
   \leq \frac{\epsilon}{5\sqrt{5}}.
%
\end{eqnarray*}
Therefore, for $x\in B$ and $t''-t'<\mu,$ $t',\,t''\in(t_{i},t_{i+1}),\,i\in\mathbf{Z},$
\begin{eqnarray*}
  &&\| (\Phi_{2}x)(t'')-(\Phi_{2}x)(t')\|_{L_{2}(P,H)}\\
   &=& \sqrt{5} \|\int_{t'}^{t''}T(t''-s)f(s,x(s),x_{s})\mathrm{d}s\|_{L_{2}(P,H)}
     + \sqrt{5}\|\int_{t'}^{t''}T(t''-s)g(s,x(s),x_{s})\mathrm{d}w(s)\|_{L_{2}(P,H)}\\&&+ \sqrt{ 5}\|\int_{-\infty}^{t'}[T(t''-s)-T(t'-s)]f(s,x(s),x_{s})\mathrm{d}s\|_{L_{2}(P,H)}\\
    && +\sqrt{5}\|\int_{-\infty}^{t'}[T(t''-s)-T(t'-s)]g(s,x(s),x_{s})\mathrm{d}w(s)\|_{L_{2}(P,H)}\\
    &&+\sqrt{5}\|\sum_{t_{i}<t'}[T(t''-t_{i})-T(t'-t_{i})]I_{i}(x(t_{i}))\|_{L_{2}(P,H)}\leq \epsilon,
\end{eqnarray*}
which shows that $\{\Phi_{2} x:x\in B\}$ is equicontinuous at each interval $(t_{i},t_{i+1})$ $(i\in\mathbf{Z}).$

Claim 2.
 $\{\Phi_{2} x:x\in B\}$ maps $B$ into a precompact set in $B.$ That is, for each fixed $t\in R,$ the set $V(t)=\{\Phi_{2}x(t):x\in B\}$ is precompact in $B.$ For each $t\in \mathbf{R},$ $0<\epsilon<1,$ $x\in B,$ define
 \begin{eqnarray*}
  \Phi_{2}^{\epsilon} (x) (t)&=&\int_{-\infty}^{t-\epsilon}T(t-s)f(s,x(s),x_{s})\mathrm{d}s\\
    \nonumber &&+\int_{-\infty}^{t-\epsilon}T(t-s)g(s,x(s),x_{s})\mathrm{d}w(s)+\sum_{t_{i}<t-\epsilon}T(t-t_{i})I_{i}(x(t_{i}))\\
   &=&  T(\epsilon)[\int_{-\infty}^{t-\epsilon}T(t-\epsilon-s)f(s,x(s),x_{s})\mathrm{d}s\\
    \nonumber &&+\int_{-\infty}^{t-\epsilon}T(t-\epsilon-s)g(s,x(s),x_{s})\mathrm{d}w(s)+\sum_{t_{i}<t-\epsilon}T(t-\epsilon-t_{i})I_{i}(x(t_{i})).]\\
   &=& T(\epsilon)   (\Phi_{2} x) (t-\epsilon).
 \end{eqnarray*}
Since $\{\Phi_{2} x:x\in B\}$ is bounded and $T(\epsilon)$ is compact, $\{\Phi_{2}^{\epsilon} x (t):x\in B\}$ is a relatively compact subset of $L^{2}(P,H).$ Moreover, for $\epsilon$ is small enough and the points $t$ and $t-\epsilon$ belong to the same interval of continuity of $x,$ we can derive
\begin{eqnarray*}
 && \Phi_{2} x (t)-\Phi_{2}^{\epsilon} x (t) = \int_{t-\epsilon}^{t}T(t-s)f(s,x(s),x_{s})\mathrm{d}s+\int_{t-\epsilon}^{t}T(t-s)g(s,x(s),x_{s})\mathrm{d}w(s)\\
 && \|\Phi_{2} x (t)-\Phi_{2}^{\epsilon} x (t)\|_{L_{2}(P,H)} \\
  &\leq& \sqrt{2}\|\int_{t-\epsilon}^{t}T(t-s)f(s,x(s),x_{s})\mathrm{d}s\|_{L_{2}(P,H)}+\sqrt{2}\|\int_{t-\epsilon}^{t}T(t-s)g(s,x(s),x_{s})\mathrm{d}w(s)\|_{L_{2}(P,H)} \\
   &\leq&   \sqrt{2}[\int_{t-\epsilon}^{t}Me^{-\delta(t-s)}\mathrm{d}s\mathbf{E}\int_{t-\epsilon}^{t}e^{-\delta(t-s)}\|f(s,x(s),x_{s})\|_{H}^{2}\mathrm{d}s]^{\frac{1}{2}}\\&&+\sqrt{2}[\mathbf{E}\int_{t-\epsilon}^{t}M^{2}e^{-2\delta(t-s)}\|g(s,x(s),x_{s})\|_{\mathcal{L}_{2}^{0}}^{2}\mathrm{d}s]^{\frac{1}{2}} \\
      &\leq&\sqrt{2} M\epsilon F_{L_{0}}+\sqrt{2}M\epsilon G_{L_{0}},
\end{eqnarray*}
so the set $V(t)=\{\Phi_{2}x(t):x\in B\}$ is precompact in $B$ for each $t\in R.$
 Since $\{\Phi_{2}x(t):x\in B\}\subseteq (PC)_{h}^{0}(\mathbf{R},L_{2}(P,H))$ and $\{\Phi_{2}x:x\in B\}$ satisfies the conditions of Lemma \ref{lem2},
the operator $\Phi_{2}$ is completely continuous. By the Krasnoselskii's fixed point theorem, we know that $\Phi$ has a fixed point $x\in B;$ that is (\ref{eq1}) has a square mean piecewise almost periodic solution $x(t).$ \end{proof}
\begin{corollary}\label{cor}
Suppose $G,\,f\in\mathcal{A}\mathcal{P}_{T}(\mathbf{R}\times L_{2}(P,H)\times\mathcal{B},L_{2}(P,H)),\,g(t,\cdot,\cdot)\in\mathcal{A}\mathcal{P}_{T}\big(\mathbf{R}\times L_{2}(P,H)\times\mathcal{B},\mathcal{L}_{2}^{0}\big).$ Condition (\ref{lipf}) and
\begin{eqnarray}\label{lipg}
  &&\|g(t,x_{1},\tilde{x_{1}})-g(t,x_{2},\tilde{x_{2}})\|_{\mathcal{L}^{0}_{2}}\leq M_{g}(\|x_{1}-x_{2}\|_{L_{2}(P,H)}+\|\tilde{x_{1}}-\tilde{x_{2}}\|_{\mathcal{B}}),
\\
&&\|I_{i}(x_{1})-I_{i}(x_{2})\|_{L_{2}(P,H)}\leq M_{I}\|x_{1}-x_{2}\|_{L_{2}(P,H)},
\end{eqnarray}
 hold, for any $t\in \mathbf{R}$ and $(x_{1},\tilde{x_{1}}),\,(x_{2},\tilde{x_{2}})\in L_{2}(P,H)\times\mathcal{B},$ where $ M_{g},\,M_{I} $ are two positive constants. Moreover $4[ \|(-A)^{-\alpha}\|+\frac{\Gamma(\alpha)}{\delta} M_{1-\alpha}]M_{G}+\frac{2\sqrt{6}MM_{f}}{\delta}+\frac{2\sqrt{3}MM_{g}}{\delta}+\frac{\sqrt{6}M_{I}}{1-e^{-\delta r}}<1,$
  then (\ref{eq1}) has a unique square mean almost periodic solution.
\end{corollary}
\begin{proof}
As the discussion in Theorem \ref{theorem}, $\Phi_{1}$ is a contracting mapping. We only need to show $\Phi_{2}$ is a contraction.
For $x,\,y\in B,$
\begin{eqnarray*}
 \| \Phi_{2}(x)-\Phi_{2}(y)\|_{L_{2}(P,H)} &=& \sqrt{3}[\mathbf{E}(\int_{-\infty}^{t}\|T(t-s)\|\|f(s,x(s),x_{s})-f(s,y(s),y_{s})\|_{H}\mathrm{d}s)^{2}]^{\frac{1}{2}} \\\nonumber&& +\sqrt{3}(\mathbf{E}\int_{-\infty}^{t}\|T(t-s)\|^{2}\|g(s,y(s),y_{s})-g(s,y(s),y_{s})\|^{2}_{\mathcal{L}_{2}^{0}}\mathrm{d}s)^{\frac{1}{2}}\\
   \nonumber&&+\sqrt{3}\|\sum_{t_{i}<t}T(t-t_{i})[I_{i}(x(t_{i}))-I_{i}(y(t_{i}))]\|_{L_{2}(P,H)}\\
   &\leq&\sqrt{3}M[\mathbf{E}\int_{-\infty}^{t}e^{-\delta(t-s)}\mathrm{d}s\int_{-\infty}^{t}e^{-\delta(t-s)}\|f(s,x(s),x_{s})-f(s,y(s),y_{s})\|_{H}^{2}\mathrm{d}s]^{\frac{1}{2}}\\
   &&\sqrt{3}M(\mathbf{E}\int_{-\infty}^{t}e^{-2\delta(t-s)}\mathrm{d}s\|g(s,y(s),y_{s})-g(s,y(s),y_{s})\|^{2}_{\mathcal{L}_{2}^{0}}\mathrm{d}s)^{\frac{1}{2}}\\
&&+\sqrt{3}M\|\sum_{t_{i}<t}e^{-\delta(t-t_{i})}\|[I_{i}(x(t_{i}))-I_{i}(y(t_{i}))]\|_{L_{2}(P,H)}\\
&\leq& [\frac{2\sqrt{3}MM_{f}}{\delta}+\frac{2\sqrt{3}MM_{g}}{\sqrt{2}\delta}+\frac{\sqrt{3}M_{I}}{1-e^{-\delta r}}]\|x-y\|_{\infty}.
\end{eqnarray*}
So when $4[ \|(-A)^{-\alpha}\|+\frac{\Gamma(\alpha)}{\delta} M_{1-\alpha}]M_{G}+\frac{2\sqrt{6}MM_{f}}{\delta}+\frac{2\sqrt{3}MM_{g}}{\delta}+\frac{\sqrt{6}M_{I}}{1-e^{-\delta r}}<1,$ by the contraction mapping principle, $\Phi$ has a unique fixed point $x(t),$ which is the  square mean almost periodic solution of (\ref{eq1}).
\end{proof}

\section{Stability}
In this section , we consider the exponential stability of the piecewise almost periodic solution of system \ref{eq1}. We first prepare a Lemma.
\begin{lemma}\cite{9} \label{lemp}
Let a nonnegative piecewise continuous function $u(t)$ satisfying for $t\geq t_{0}$ the inequality
\begin{equation*}
  u(t)\leq C+\int_{t_{0}}^{t}v(\tau)u(\tau)\mathrm{d}\tau+\sum_{t_{0}<\tau_{i}<t}\beta_{i}u(\tau_{i}),
\end{equation*}
where $C\geq 0,$ $\beta_{i}\geq 0,$ $v(\tau)>0,$ and $\tau_{i}',\,i=1,2,\ldots$ are discontinuity points of first type of the function $u(t).$ Then the following estimate holds,
\begin{equation*}
  u(t)\leq C\prod_{t_{0}<\tau_{i}<t}(1+\beta_{i})e^{\int_{t_{0}}^{t}v(\tau)\mathrm{d}\tau}.
\end{equation*}
\end{lemma}
\begin{theorem}
Assume the conditions of Corollary \ref{cor} are fulfilled and
 \begin{equation*}
  \frac{1}{\gamma}\ln(1+\frac{2M^{2}M_{I}^{2}}{\Lambda(1-e^{-\delta r})})+\frac{4M_{f}^{2}M^{2}}{\Lambda\delta}+\frac{4M_{g}^{2}M^{2}}{\Lambda}-\delta<0.
\end{equation*}
Then system (\ref{eq1}) has an exponentially stable almost periodic solution.
\end{theorem}
\begin{proof}
%
 Let $x(t)=x(t,\sigma,\varphi)$  and $y(t)=y(t,\sigma,\psi)$ be two solutions of equation \ref{eq1}, then
by (\ref{sol}) and the Cauchy-Schwarz inequality, we can derive that
\begin{eqnarray*}
 &&\mathbf{E}\|x(t)-y(t)\|_{H}^{2} \\
 &\leq&6 \mathbf{E} \|T(t)[(\varphi(0)-\psi(0))-(G(\sigma,\varphi(0),\varphi)-G(\sigma,\psi(0),\psi))]\|_{H}^{2}\\
   &&+6\mathbf{E}\|G(t,x(t),x_{t})-G(t,y(t),y_{t})\|_{H}^{2} +6\mathbf{E}\|\int_{\sigma}^{t}A T(t-s)(G(s,x(s),x_{s})-G(s,y(s),y_{s}))\mathrm{d}s\|_{H}^{2}  \\
   &&+ 6\mathbf{E}\|\int_{\sigma}^{t}T(t-s)[f(s,x(s),x_{s})-f(s,y(s),y_{s})]\mathrm{d}s\|_{H}^{2}\\
   &&+6 \mathbf{E}\|\int_{\sigma}^{t}T(t-s)[g(s,x(s),x_{s})-g(s,y(s),y_{s})]\mathrm{d}w(s)\|_{H}^{2}\\
   &&+6 \mathbf{E}\|\sum_{\sigma<t_{i}<t}T(t-t_{i})[I_{i}(x(t_{i}))-I_{i}(y(t_{i}))]\|_{H}^{2}.
\end{eqnarray*}
We note that
\begin{eqnarray*}
&&\mathbf{E} \|T(t)[(\varphi(0)-\psi(0))-(G(\sigma,\varphi(0),\varphi)-G(\sigma,\psi(0),\psi))]\|_{H}^{2}\\
&\leq& M^{2}e^{-2\delta t}[2\mathbf{E} \|\varphi(0)-\psi(0)\|_{H}^{2}+4M_{G}^{2}\|(-A)^{-\alpha}\|_{H}^{2}\mathbf{E} (\|\varphi(0)-\psi(0)\|^{2}_{H}+\|\varphi-\psi\|_{\mathcal{B}}^{2})\\
&\leq&[2 M^{2}e^{-2\delta t}+8 M^{2}M_{G}^{2}e^{-2\delta t}\|(-A)^{-\alpha}\|_{H}^{2}]\sup_{s\leq 0}\mathbf{E}\|\varphi(s)-\psi(s)\|^{2}_{H},
\end{eqnarray*}
\begin{eqnarray*}
  &&\mathbf{E}\|G(t,x(t),x_{t})-G(t,y(t),y_{t})\|_{H}^{2} +\mathbf{E}\|\int_{0}^{t}A T(t-s)(G(s,x(s),x_{s})-G(s,y(s),y_{s}))\mathrm{d}s\|_{H}^{2} \\
  &\leq& [\|(-A)^{-\alpha}\|^{2}+\frac{M_{1-\alpha}^{2}\Gamma^{2}(\alpha)}{\delta^{2}}]4 M_{G}^{2}\sup_{s\leq t}\mathbf{E}\|x(t)-y(t)\|_{H}^{2},
\end{eqnarray*}
\begin{eqnarray*}
 && \mathbf{E}\|\int_{\sigma}^{t}T(t-s)[f(s,x(s),x_{s})-f(s,y(s),y_{s})]\mathrm{d}s\|_{H}^{2}\\&&+\mathbf{E}\int_{\sigma}^{t} \|T(t-s)\|\|[g(s,x(s),x_{s})-g(s,y(s),y_{s})]\|_{\mathcal{L}_{2}^{0}}^{2}\mathrm{d}s  \\
   &\leq&\int_{\sigma}^{t}M^{2}e^{-\delta(t-s)}\mathrm{d}s\int_{0}^{t}e^{-\delta(t-s)}\mathbf{E}\|f(s,x(s),x_{s})-f(s,y(s),y_{s})\|_{H}^{2}\mathrm{d}s\\&&
   +\int_{\sigma}^{t}M^{2}e^{-2\delta(t-s)}\mathbf{E}\|g(s,x(s),x_{s})-g(s,y(s),y_{s})\|_{\mathcal{L}_{2}^{0}}^{2}\mathrm{d}s  \\
 &\leq&[\frac{2M_{f}^{2}M^{2}}{\delta}+4M^{2}M_{g}^{2}]\int_{\sigma}^{t}e^{-\delta(t-s)}\sup_{s\leq t}\mathbf{E}\|x(s)-y(s)\|_{H}^{2}\mathrm{d}s,
\end{eqnarray*}
\begin{eqnarray*}
 && \mathbf{E}\|\sum_{\sigma<t_{i}<t}T(t-t_{i})[I_{i}(x(t_{i}))-I_{i}(y(t_{i}))]\|_{H}^{2}\\
  &\leq& \mathbf{E}[\sum_{\sigma<t_{i}<t}Me^{-\delta(t-t_{i})}\|I_{i}(x(t_{i}))-I_{i}(y(t_{i}))\|_{H}]^{2} \\
  &\leq&(\sum_{\sigma<t_{i}<t}M^{2}e^{-\delta(t-t_{i})})\sum_{\sigma<t_{i}<t}e^{-\delta(t-t_{i})}\mathbf{E}\|I_{i}(x(t_{i}))-I_{i}(y(t_{i}))\|_{H}^{2}\\
  &\leq&\frac{2M^{2}M_{I}^{2}}{1-e^{-\delta r}}\sum_{\sigma<t_{i}<t}e^{-\delta(t-t_{i})}\mathbf{E}\|x(t_{i})-y(t_{i})\|_{H}^{2},
\end{eqnarray*}
then
\begin{eqnarray*}
&&[1-[\|(-A)^{-\alpha}\|^{2}+\frac{M_{1-\alpha}^{2}\Gamma^{2}(\alpha)}{\delta^{2}}]24 M_{G}^{2}]\mathbf{E}\|x(t)-y(t)\|_{H}^{2}  \\
 &\leq&[2M^{2}e^{-2\delta t}+4M^{2}M_{G}e^{-2\delta t}\|(-A)^{-\alpha}\|_{H}^{2}]\sup_{\theta\leq0}\mathbf{E}\|\varphi(\theta)-\psi(\theta)\|^{2}_{H} \\
   &&  +[4\frac{M_{f}^{2}M^{2}}{\delta}+4M_{g}^{2}M^{2}]\int_{\sigma}^{t}e^{-\delta(t-s)}\sup_{\theta\leq s}\mathbf{E}\|x(\theta)-y(\theta)\|_{H}^{2}\mathrm{d}s\\&&+\frac{2M^{2}M_{I}^{2}}{1-e^{-\delta r}}\sum_{\sigma<t_{i}<t}e^{-\delta(t-t_{i})}\mathbf{E}\|x(t_{i})-y(t_{i})\|_{H}^{2}.
\end{eqnarray*}
Let $\Lambda=1-[\|(-A)^{-\alpha}\|^{2}+\frac{M_{1-\alpha}^{2}\Gamma^{2}(\alpha)}{\delta^{2}}]24 M_{G}^{2},$
consequently
\begin{eqnarray*}
  e^{\delta t}\sup_{s\leq t}\mathbf{E}\|x(t)-y(t)\|_{H}^{2} &\leq& \frac{2M^{2}+8M^{2}M_{G}^{2}\|(-A)^{-\alpha}\|_{H}^{2}}{\Lambda}\sup_{\theta\leq0}\mathbf{E}\|\varphi(\theta)-\psi(\theta)\|^{2}_{H} \\
 &&  +[\frac{4M_{f}^{2}M^{2}}{\Lambda\delta}+\frac{4M_{g}^{2}M^{2}}{\Lambda}]\int_{\sigma}^{t}e^{\delta s}\sup_{\theta\leq s}\mathbf{E}\|x(\theta)-y(\theta)\|_{H}^{2}\mathrm{d}s\\&&+\frac{2M^{2}M_{I}^{2}}{\Lambda(1-e^{-\delta\gamma})}\sum_{\sigma<t_{i}<t}e^{\delta t_{i}}\mathbf{E}\|x(t_{i})-y(t_{i})\|_{H}^{2}
\end{eqnarray*}
Let $\gamma (t)= e^{\delta t}\sup_{s\leq t}\mathbf{E}\|x(t)-y(t)\|_{H}^{2},$ then
\begin{eqnarray*}
\gamma (t) &\leq& \frac{2M^{2}+8M^{2}M_{G}^{2}\|(-A)^{-\alpha}\|_{H}^{2}}{\Lambda} \gamma(0)\\
  &&+[\frac{4M_{f}^{2}M^{2}}{\Lambda\delta}+\frac{4M_{g}^{2}M^{2}}{\Lambda}]\int_{\sigma}^{t}\gamma (s)\mathrm{d}s+\frac{2M^{2}M_{I}^{2}}{\Lambda(1-e^{-\delta r})}\sum_{\sigma<t_{i}<t}\gamma (t_{i}).
\end{eqnarray*}
Hence by Lemma\ref{lemp}, we can show that
\begin{eqnarray*}
 \gamma (t) &\leq& \frac{2M^{2}+8M^{2}M_{G}\|(-A)^{-\alpha}\|_{H}^{2}}{\Lambda} \gamma(0)\Pi_{\sigma<t_{i}<t}(1+\frac{2M^{2}M_{I}^{2}}{\Lambda(1-e^{-\delta\gamma})})e^{\int_{\sigma}^{t}(\frac{4M_{f}^{2}M^{2}}{\Lambda\delta}+\frac{4M_{g}^{2}M^{2}}{\Lambda})\mathrm{d}s}\\
   &\leq& \frac{2M^{2}+8M^{2}M_{G}\|(-A)^{-\alpha}\|_{H}^{2}}{\Lambda} \gamma(0)(1+\frac{2M^{2}M_{I}^{2}}{\Lambda(1-e^{-\delta r})})^{\frac{t}{\gamma}}e^{(\frac{4M_{f}^{2}M^{2}}{\Lambda\delta}+\frac{4M_{g}^{2}M^{2}}{\Lambda})t} \\ &=&  \frac{2M^{2}+8M^{2}M_{G}\|(-A)^{-\alpha}\|_{H}^{2}}{\Lambda} \gamma(0)e^{[\frac{1}{\gamma}\ln(1+2\frac{M^{2}M_{I}^{2}}{\Lambda(1-e^{-\delta r})})+\frac{4M_{f}M^{2}}{\Lambda\delta}+\frac{4M_{g}^{2}M^{2}}{\Lambda}]t}.
\end{eqnarray*}
that is
\begin{eqnarray*}
 && \sup_{s\leq t}\mathbf{E}\|x(t)-y(t)\|_{H}^{2}\\
 &\leq& \frac{2M^{2}+8M^{2}M_{G}^{2}\|(-A)^{-\alpha}\|_{H}^{2}}{\Lambda} \gamma(0)e^{[\frac{1}{\gamma}\ln(1+\frac{2M^{2}M_{I}^{2}}{\Lambda(1-e^{-\delta r})})+\frac{4M_{f}^{2}M^{2}}{\Lambda\delta}+\frac{4M_{g}^{2}M^{2}}{\Lambda}-\delta]t}.
\end{eqnarray*}
Since $\frac{1}{\gamma}\ln(1+\frac{2M^{2}M_{I}^{2}}{\Lambda(1-e^{-\delta r})})+\frac{4M_{f}^{2}M^{2}}{\Lambda\delta}+\frac{4M_{g}^{2}M^{2}}{\Lambda}-\delta<0,$ the square mean piecewise almost periodic solution of system (\ref{eq1}) is exponentially stable.
\end{proof}

\end{document}